\documentclass[11pt]{amsart}
\usepackage{geometry} % see geometry.pdf on how to lay out the page. There's lots.
\usepackage{pgfplots}
\usepackage{psfrag}
\usepackage{amsmath}
\usepackage{calc}
\usepackage{systeme}
\usepackage{amsfonts}
\usepackage{amsthm}
\usepackage{algorithm2e}
\usepackage[applemac]{inputenc}
\usepackage{bbold}
\usepackage[numbers]{natbib}

\author{%
  Charles Bertucci $^1$, Jean-Michel Lasry $^2$, %\footrecall{alley}%
 Pierre-Louis Lions$^{2,3}$
  }

\newtheorem{Theorem}{Theorem}
\newtheorem{Lemma}{Lemma}
\newtheorem{Rem}{Remark}
\newtheorem{Def}{Definition}
\newtheorem{Prop}{Proposition}

\usepackage[foot]{amsaddr}

\title{Master equation for the finite state space planning problem}
\thanks{$^1$ : CMAP, Ecole Polytechnique, UMR 7641, 91120 Palaiseau, France\\
$^2$ :Universit\'e Paris-Dauphine, PSL Research University,UMR 7534, CEREMADE, 75016 Paris, France\\
$^3$ : Coll\`ege de France, 3 rue d'Ulm, 75005, Paris, France
}
\date{} % delete this line to display the current date

\begin{document}
\maketitle
\begin{abstract}
We present results of existence, regularity and uniqueness of solutions of the master equation associated with the mean field planning problem in the finite state space case, in the presence of a common noise. The results hold under monotonicity assumptions, which are used crucially in the different proofs of the paper. We also make a link with the trajectories induced by the solution of the master equation and start a discussion on the case of boundary conditions.
\end{abstract}
\tableofcontents
\section*{Introduction}
This paper is concerned with the study of the mean field games planning problem in finite state space, in the presence of common noise. The cases of a continuous state space and of generalized optimal transport in the presence of a common noise shall be treated in a future work.

\subsection*{General introduction}
The planning problem is a mathematical formulation of a model in which an infinite number of non-atomic and identical players face a time dependent game in which the final distribution of players is constrained. The latter constraint is of course non-classical and raises different modeling problems. The most serious one being that the final constraint on the distribution of players cannot be impose on those players, which, being non-atomic, cannot affect the distribution. Formally, the planning problem addresses the previous issue as follows : there are strong incentives in the cost functions of the players to induce a behavior which is consistent with the constraint.

The planning problem originates from the work of the last two authors and was first presented in the lectures of the third author at Coll\`ege de France \citep{lions2007cours}. When the noises affecting the players are independent and when the state space is continuous, a system of PDE has been established in \citep{lions2007cours} to characterize the equilibria of the planning problem. The typical form of this system is 
\begin{equation}\label{ppsystem}
\begin{cases}
-\partial_t u - \nu \Delta u + H(x, \nabla u, m) = 0 \text{ in } (0,T) \times \mathbb{R}^d;\\
\partial_t m - \nu \Delta m - \text{div}(D_pH(x, \nabla u, m)m) = 0 \text{ in } (0,T) \times \mathbb{R}^d;\\
m(0) = m_1 ; m(T) = m_2 \text{ in } \mathbb{R}^d.
\end{cases}
\end{equation}
In (\ref{ppsystem}), $u$ denotes the value function of the players, $m$ their density, $H(x,p,z)$ is the Hamiltonian associated to the game and $\nu \geq 0$ is a parameter describing the intensity of the noise. A subclass of the planning problem of the utmost importance is of course the optimal transport problem. The analogy is almost transparent when comparing (\ref{ppsystem}) with the so-called Benamou-Brenier formulation of the optimal transport problem \citep{benamou2000computational}. Let us comment on the fact that from a modeling point of view, the main difference between the mean field planning problem and optimal transport, is that optimal transport is associated to a centralize optimization problem, whereas the mean field planning problem deals with a game in which non-atomic players, or infinitesimal masses, are supposed to have a behavior which is individual.\\

The system (\ref{ppsystem}) has now been quite extensively studied, especially by analogy to his mean field game (MFG for short) counterpart. The point of view being to impose the terminal condition 
\begin{equation}\label{penalcont}
u(T) = \frac{1}{\epsilon}(m(T) - m_2) \text{ in } \mathbb{R}^d
\end{equation}
instead of the condition $m(T) = m_2$, and to let $\epsilon$ goes to $0$. The first results on (\ref{ppsystem}) dates back from \citep{lions2007cours} and concern either the deterministic case $\nu = 0$ or the quadratic one $H(x,p,z) = |p|^2 - f(z)$. Later on, a more general case has been studied in \citep{porretta2014planning} with the notion of weak solutions. Numerical methods for this type of system has been presented in \citep{achdou2012mean}. More recently, the works \citep{graber2019planning,orrieri2019variational} studied (\ref{ppsystem}) in the potential case. In the so-called potential approach, (\ref{ppsystem}) is interpreted as the optimality conditions of an infinite dimensional optimal control problem. As the aim of this paper is to study the case of a common noise, we cannot generalize the approach which consists as looking at a forward-backward system similar to (\ref{ppsystem}), instead we have to study the master equation associated to planning problem.\\
\\

\subsection*{The master equation approach}
In this paper we study a different setting than the one we just mentioned. We place ourselves in the finite state space case in which the study of the master equation is less difficult and less technical than the continuous case, that we shall address in a future work. In this setting, a typical form taken by the master equation with $d$ states in the MFG context is :
\begin{equation}\label{finitemfg}
\begin{cases}
\partial_t U + (F(x,U)\cdot \nabla_x)U + \lambda (U - (DT)^*U(Tx)) = G(x, U) \text{ in } (0,T)\times\mathbb{R}^d;\\
U(0,x) = U_0(x) \text{ in } \mathbb{R}^d;
\end{cases}
\end{equation}
where, to simplify notations time has been reversed, $F$ and $G$ are mappings from $\mathbb{R}^{2d}$ onto $\mathbb{R}^d$ which describes respectively the evolution of the distribution of players and the costs paid by the players, $T$ is a mapping from $\mathbb{R}^d$ into itself which describes the common noise and $\lambda \geq 0$ describes the intensity of the common noise. The solution $U = (U^i)_{1 \leq i \leq d}:(0,T)\times \mathbb{R}^d\to \mathbb{R}^d$ is the value function of the players, meaning that $U^i(t,x)$ is the value for a generic player in state $i$ when the distribution of players in the $d$ states is given by $x\in \mathbb{R}^d$ and there remains a time interval of length $t$ in the game.

In this paper we want to analyse the solution of (\ref{finitemfg}) when the initial condition has been replaced by a suitable penalization, by analogy with (\ref{penalcont}), when this penalization goes to infinity. We shall show that this approach leads to the construction of the value function of the players in the planning problem, solution of a master equation. We expect a very singular behavior for the solution of the planning problem's master equation at $t= 0$ and we show in this paper how to characterize precisely this singularity.

In the classical MFG setting, i.e. when there is no blow-up at the time $t= 0$, the system (\ref{finitemfg}) has been studied. Since its introduction by the last two authors, their results of well-posedness in the case $\lambda = 0$ have been presented in \citep{lions2007cours}. The case of common noise, i.e. $\lambda > 0$, has been treated in \citep{bertucci2019some}. Let us also note that other type of noise have been studied in \citep{bayraktar2019finite}, which contrary to the noise we study here, do not propagate monotonicity but yield some kind of ellipticity for the master equation.

\subsection*{Structure of the paper}
The rest of the paper is structured as follows. In section 1 we prove some key estimates which allow us to pass to the limit in a penalized master equation. In section 2 we show that we can characterize the solution of the master equation associated to the planning problem by using the Yosida's regularization when this master equation is posed on $\mathbb{R}^d$. In section 3 we show that the techniques of section 2 can be adapted to a case in which the master equation is posed on a subdomain of $\mathbb{R}^d$.

\section{Notations and preliminary results}
\subsection{Notations}
We begin by with some notations and definitions.
\begin{itemize}
\item We shall use the notations $|\cdot|$ for the norm of vectors and $\|\cdot\|$ for the norm of operator.
\item The identity of applications is denoted $Id$ and the identity of matrices $I$.
\item An application $A$ from an Hilbert space $(H,\langle , \rangle)$ into itself is $\alpha$-monotone ($\alpha \geq 0$) if for all $x,y \in H$
\[
\langle A(x) - A(y), x-y \rangle \geq \alpha \langle x-y, x-y\rangle.
\]
\item An application $A$ is simply called monotone if the previous inequality holds with $\alpha = 0$.
\item For any element $x$ of an Hilbert space $H$, $A_x$ denotes the multivalued operator whose domain is $\{x\}$ and such that $A(x) = H$.
\item The Yosida approximation $S_{\delta}A$ of parameter $\delta$ of a monotone operator $A$ is defined by 
\[
S_{\delta}A = A\circ (Id + \delta A )^{-1}.
\]
\end{itemize}
\subsection{Preliminary results}
The first result we present in this section is a result of existence of solutions of (\ref{finitemfg}). This result is fairly simple, once some a priori estimates has been established as was already mentioned in \citep{lions2007cours}. However, as the question of existence is central in this paper, we wish to state and prove a specific result for the sake of completeness.
\begin{Prop}\label{existmfg}
Assume that 
\begin{itemize}
\item $F$, $G$ and $U_0$ are lipschitz applications.
\item $(G,F) : \mathbb{R}^{2d} \to \mathbb{R}^{2d}$ and $U_0 : \mathbb{R}^d \to \mathbb{R}^d$ are monotone.
\item either $(G,F)$ is $\alpha$ monotone in its first variable and $U_0$ is $\alpha$ monotone or $(G,F)$ is $\alpha$ monotone in its second variable for some $\alpha > 0$.
\item either $\lambda = 0$ or $T$ is an affine application.
\end{itemize}
Then, there exists a unique function $U$ defined on $\mathbb{R}_+\times \mathbb{R}^d$, solution of (\ref{finitemfg}) which is monotone and lipschitz in the $x$ variable.
\end{Prop}
\begin{proof}
Under the standing assumptions, for any $T> 0$, there exists $C > 0$, such that for any $\epsilon > 0$, for any smooth solution $U$ of 
\begin{equation}\label{regfinitemfg}
\begin{cases}
\partial_t U - \epsilon \Delta U + (F(x,U)\cdot \nabla_x)U + \lambda (U - (DT)^*U(Tx)) = G(x, U) \text{ in } (0,T)\times\mathbb{R}^d;\\
U(0,x) = U_0(x) \text{ in } \mathbb{R}^d;
\end{cases}\end{equation}
the following holds
\begin{equation}\label{est1}
\forall t \leq T, \forall x \in \mathbb{R}^d, \|D_x U(t,x)\| \leq C.
\end{equation}
This results is a direct generalization of the estimates established in \citep{lions2007cours,bertucci2019some} which is obtained by remarking that adding the term $-\epsilon \Delta$ leaves the proof of the result unchanged. Let us fix $\epsilon > 0$ and let us define the operator $\mathcal{F}$ from the set of locally lipschitz function into itself by $V = \mathcal{F}(U)$ is the unique solution of
\begin{equation}
\begin{cases}
\partial_t V - \epsilon \Delta V + (F(x,U)\cdot \nabla_x)V + \lambda (V - (DT)^*U(Tx)) = G(x, U) \text{ in } (0,T)\times\mathbb{R}^d;\\
U(0,x) = U_0(x) \text{ in } \mathbb{R}^d.
\end{cases}\end{equation}
From chapter IV, paragraph 11 in \citep{lady}, the operator $\mathcal{F}$ is continuous and compact. Assume that there exists $µ\in (0,1)$ and $U$ such that $µ\mathcal{F}(U) = U$. Then $U$ solves 
\begin{equation}
\begin{cases}
\partial_t U - \epsilon \Delta U + (µF(x,U)\cdot \nabla_x)U + \lambda(1 - µ)U + µ\lambda (U - (DT)^*U(Tx)) = µG(x, U) \text{ in } (0,T)\times\mathbb{R}^d;\\
U(0,x) = U_0(x) \text{ in } \mathbb{R}^d.
\end{cases}\end{equation}
Hence, from (\ref{est1}) we know that $U$ is bounded by a constant independent of $µ$ in some H\"older space. Thus $\mathcal{F}$ has a fixed point $U_{\epsilon}$ for any $\epsilon > 0$. We now want to show some compactness on the sequence $(U_{\epsilon})_{\epsilon > 0}$ of solutions of $(\ref{regfinitemfg})$. The gradient estimate (\ref{est1}) of course holds and we only need to prove some estimate on the time derivative of the solutions. We first show that for all $\epsilon >0$ there exists $C_1 > 0, \eta > 0$ such that for $t$ sufficiently small, for all $x \in \mathbb{R}^d$
\[
|U_{\epsilon}(t,x) - U_0(x)| \leq C_1t (1 + |x|) + \eta.
\]
We begin by showing the previous inequality component by component. Let us assume that such an inequality does not hold. Then for any $C_1> 0, \eta > 0$, there exists $t_0 >0, x_0 \in \mathbb{R}^d, i \in \{1;...;d\}$ with $t_0$ sufficiently small such that 
\[
\begin{cases}
U^i_{\epsilon}(t_0,x_0) = U^i_0(x) + C_1t_0 (1 + |x_0|) + \eta;\\
\partial_t U^i_{\epsilon}(t_0,x_0) \geq C_1(1 + |x_0|);\\
\nabla_x U^i_{\epsilon}(t_0,x_0) = \nabla_xU^i_0(x_0) + C_1 t_0;\\
-\Delta U^u_{\epsilon}(t_0,x_0) \geq 0.
\end{cases}
\]
Thus evaluating (\ref{regfinitemfg}) at $(t_0,x_0)$ we deduce
\[
C_1 ( 1 + |x_0|) \leq \delta(1 + |x_0| + |U_{\epsilon}(t_0,x_0)|)
\]
for some $\delta > 0$ which depends only on $F,G,\lambda, T$ and the constant $C$ from (\ref{est1}), in particular it does not depend on $\eta$ or $C_1$. From which we easily deduce that there exists $C_1> 0$, such that for all $\eta > 0$, 
\[
|U_{\epsilon}(t,x) - U_0(x)| \leq C_1t (1 + |x|) + \eta.
\]
From which we deduce that it is also true for $\eta = 0$. Therefore we obtain
\[
|\partial_t U_{\epsilon}(0,x)| \leq C_1 (1 + |x|).
\]
From this we deduce that the sequence $(U_{\epsilon})_{\epsilon >0}$ is locally in $[0,\infty)\times \mathbb{R}^d$, uniformly lipschitz in $\epsilon$. From this uniform continuity and Ascoli-Arzela theorem, we know that along a subsequence, $(U_{\epsilon})_{\epsilon > 0}$ converges uniformly toward a locally lipschitz function $U$, solution of (\ref{finitemfg}).
\end{proof}
We now state a regularizing result associated to the structure of the system we are studying. To our knowledge, this is the first result of this type on a master equation.
\begin{Theorem}\label{effetreg}
Let $U$ be a solution of (\ref{finitemfg}). Assume that
\begin{itemize}
\item $U_0$ and $(G,F)$ are monotone.
\item $F$ and $G$ are globally lipschitz.
\item $(G,F)$ is $\alpha$ monotone in its second argument for some $\alpha > 0$.
\item $T$ is affine.
\end{itemize}
Then there exists a constant $C > 0$ (independent of $U_0$) such that for all $t\leq 1$
\begin{equation}\label{est2}
\|D_x U(t)\| \leq \frac{C}{t}.
\end{equation}
\end{Theorem}
\begin{Rem}
Let us insist on the fact that no assumption is made on the regularity of $U_0= U(0)$.
\end{Rem}
\begin{proof}
This proof starts by following the same argument as in \citep{lions2007cours,bertucci2019some}. We note $T = S + e$ with $S$ a linear map. We take two functions $\beta : \mathbb{R}_+ \to \mathbb{R}_+$ and $\gamma : \mathbb{R}_+ \to \mathbb{R}_+$ to be defined later on and we introduce the auxiliary function $Z_{\beta,\gamma}$ defined by 
\[
Z_{\beta}(t,x,\xi) = \langle\xi, \nabla_x W(t,x,\xi)\rangle - \beta(t) |\nabla_x W(t,x,\xi)|^2 + \gamma(t)|\xi|^2;
\]
where $W$ is defined by
\[
W (t,x,\xi) = \langle U(t,x) , \xi \rangle.
\]
The chain rule yields that $Z_{\beta,\gamma}$ satisfies the PDE :
\begin{equation}
\begin{aligned}
& \partial_t Z_{\beta, \gamma} + \langle F(x,\nabla_{\xi} W), \nabla_x Z_{\beta, \gamma}\rangle + \langle D_p F (x,\nabla_{\xi} W ) \nabla_{\xi}Z_{\beta, \gamma} , \nabla_{x} W\rangle - \langle D_p G(x, \nabla_{\xi}W)\nabla_{\xi}Z_{\beta, \gamma}, \xi\rangle \\
&+  \lambda(Z_{\beta, \gamma} - Z_{\beta, \gamma}(t,Tx, T\xi - e) )\\
& = \langle D_xG(x, \nabla_x W) \xi, \xi\rangle - \langle D_pG(x, \nabla_{\xi}W) \nabla_x W, \xi\rangle - \langle D_x F(x, \nabla_{\xi} W) \nabla_x W, \xi\rangle \\
&+ \langle D_p F(x, \nabla_{\xi}W) \nabla_x W, \nabla_x W\rangle\\
&- 2 \beta \langle D_x G(x, \nabla_x W)\xi, \nabla_x W\rangle + 2 \beta \langle D_x F(x, \nabla_{\xi} W) \nabla_x W, \nabla_x W\rangle\\
& + \beta \lambda \bigg{(} |\nabla_x W|^2 - 2 \langle \nabla_x W(t, Tx, T \xi - e), S \nabla_x W\rangle +  | \nabla_x W(t,Tx, T \xi -e)|^2 \bigg{)} \\
&- \frac{d}{dt} \beta |\nabla_x W|^2  + 2 \gamma(\langle D_p F (x,\nabla_{\xi} W ) \xi , \nabla_{x} W\rangle - \langle D_p G(x, \nabla_{\xi}W)\xi, \xi\rangle)\\
& + \frac{d}{dt}\gamma |\xi|^2 + \lambda \gamma (|\xi|^2 - |T\xi - e|^2).
\end{aligned}
\end{equation}
Using the monotonicity assumption on $(G,F)$ and the lipschitz assumptions, we deduce
\begin{equation}\label{ineqz3}
\begin{aligned}
& \partial_t Z_{\beta, \gamma} + \langle F(x,\nabla_{\xi} W), \nabla_x Z_{\beta, \gamma} \rangle + \langle D_p F (x,\nabla_{\xi} W ) \nabla_{\xi}Z_{\beta, \gamma} , \nabla_{x} W\rangle - \langle D_p G(x, \nabla_{\xi}W)\nabla_{\xi}Z_{\beta, \gamma}, \xi\rangle\\
&+  \lambda(Z_{\beta, \gamma} - Z_{\beta, \gamma}(t,Tx, S\xi ) )\\
& \geq \alpha | \nabla_x W |^2 -  \beta \| D_x G\| \cdot |\xi|^2 - \beta \| D_xG \| \cdot |\nabla_x W|^2 - 2 \beta \| D_x F( \nabla_{\xi} W)\| \cdot |\nabla_x W|^2\\
& + \beta \lambda ( |\nabla_x W|^2 - |S\nabla_x W|^2 ) - \frac{d}{dt} \beta |\nabla_x W|^2 - \gamma (\|D_pF\| - 2\|D_p G\|) |\xi|^2\\
& - \gamma \|D_p F\|\cdot |\nabla_x W |^2 + \frac{d}{dt}\gamma |\xi|^2 + \lambda \gamma (|\xi|^2 - |S\xi |^2).
\end{aligned}
\end{equation}
The right hand side of the previous equation is positive if $\gamma$ and $\beta$ satisfy 
\begin{equation}\label{condbg}
\begin{cases}
\alpha - \beta [\|D_x G\| + 2 \| D_x F\| - \lambda (1 - \|S\|^2)] - \frac{d}{dt}\beta - \gamma \|D_pF\| \geq 0 ;\\
\frac{d}{dt} \gamma + \gamma[ \lambda(1 - \|S\|^2) - \|D_pF\| - 2 \|D_pG\|] - \beta \|D_x G\| \geq 0.
\end{cases}
\end{equation}
Let us now define $\beta$ and $\gamma$ with
\begin{equation}\label{defbg}
\begin{cases}
\beta(t) = \frac{\alpha}{2}t;\\
\gamma(t) = \|D_xG\|\alpha t^2.
\end{cases}
\end{equation}
Let us now observe that there exists $t_f > 0$ such that if $\beta$ and $\gamma$ are defined by (\ref{defbg}), then (\ref{condbg}) holds for all time $t\leq t_f$. Let us now remark that at $t = 0$, $Z_{\beta,\gamma}$ satisfies 
\[
Z_{\beta,\gamma} (t,x,\xi) = \langle \xi, D_x U_0 \xi \rangle \geq 0;
\]
from the monotonicity assumption we made on $U_0$. Thus we deduce from lemma $3$ in appendix of \citep{bertucci2019some}, that $Z_{\beta, \gamma}$ stays positive for all time $t\leq t_f$. From this we deduce that for all $t\leq t_f, \xi \ne 0$ :
\[
\beta(t) \frac{|\nabla_x W|^2}{|\xi|^2} \leq \gamma(t) + \frac{|\nabla_x W|}{|\xi|}.
\]
Hence we obtain
\[
\frac{|\nabla_x W|}{|\xi|} \leq \frac{\sqrt{1 + 4 \beta(t)\gamma(t)}}{\beta(t)};
\]
from which we derive
\[
\|D_xU(t)\| \leq \frac{\sqrt{1 + 4 \beta(t)\gamma(t)}}{\beta(t)}.
\]
The result then easily follows.
\end{proof}

\section{Planning problem master equation in $\mathbb{R}^d$}
\subsection{Statement of the problem}
In this section we show that, under some assumptions, the sequence $(U_{\epsilon})_{\epsilon > 0}$ of solutions of 
\begin{equation}\label{penalpp}
\begin{cases}
\partial_t U + (F(x,U)\cdot \nabla_x)U + \lambda (U - (DT)^*U(Tx)) = G(x, U) \text{ in } (0,T)\times\mathbb{R}^d;\\
U(0,x) = \frac{1}{\epsilon}(x - x_0) \text{ in } \mathbb{R}^d;
\end{cases}
\end{equation}
converges toward a function $U$, which can be interpreted as the value function of a generic player for the mean field planning problem described by $F,G, \lambda$ and $T$ and constrained in $x_0$ at the final time. Before presenting the proof of this result, let us recall briefly the interpretation in terms of modeling of the master equation. In this model, as already mentioned, $U$ represents the value function of the players. In the case $\lambda = 0$, given this value function, the discrete density of players $x$ is assumed to evolve through the term $F$. This means that a particular density starting from $x_1$ at time $t_1$ evolves according to (time has been reversed) :
\begin{equation}
\begin{cases}
\frac{d}{dt}x(t) = F(x(t),U(t,x(t))) \text{ for } 0 \leq t\leq t_1 ;\\ x(t_1) = x_1.
\end{cases}
\end{equation}
The evolution of the value function of the players is given by the function $G$ and the terminal cost is $U(0, \cdot)$, this means that it satisfies
 \begin{equation}
\begin{cases}
\frac{d}{dt}U(t,x(t)) = G(x(t),U(t,x(t))) \text{ for } 0 \leq t\leq t_1 ; \\U(0,x(0)) = \frac{1}{\epsilon}(x(0) - x_0).
\end{cases}
\end{equation}
The main idea of the penalization term is that it induces sufficient incentives so that the final density $x(0)$ shall be close to $x_0$.\\

In the case $\lambda > 0$, we model a situation in which at random times given by a poisson process of intensity $\lambda$, all the population is affected by the transformation $T$, we refer to \citep{bertucci2019some} for more details on this type of noise.

\subsection{Properties of the Yosida approximation}
In this section, we present an argument which explains how we can pass to the limit in (\ref{penalpp}). Namely we show some compactness on the sequence of Yosida approximation of the solution of the penalized problem. Let us mention that the use of the Yosida approximation may seem arbitrary but the definition of solution that we give in the next section does not involve the choice of this approximation. Also let us mention that we could have avoided the use of Yosida approximations and worked only with the regularizing result established in the first section but we believe the present approach is more instructive. 

From proposition \ref{existmfg}, we know that $(U_{\epsilon})_{\epsilon > 0}$ is a well defined sequence of locally lipschitz functions which are monotone in space for all time. For any $\epsilon > 0, \delta > 0$, we define $V_{\epsilon,\delta} = S_{\delta}U_{\epsilon}$, the Yosida regularization of $U_{\epsilon}$ of parameter $\delta > 0$. Equivalently, we could have defined $V_{\delta, \epsilon}$ by $V_{\delta, \epsilon}(t,x) = W_{\epsilon,t}(s,x)|_{s = \delta}$ where $W$ is the solution of
\begin{equation}\label{eqYosida}
\begin{cases}
\partial_s W_{\epsilon,t} + W_{\epsilon,t}\cdot \nabla_x W_{\epsilon,t} = 0;\\
W_{\epsilon,t}(0,x) = U_{\epsilon}(t,x).
\end{cases}
\end{equation}
We present the PDE satisfied by $V_{\delta,\epsilon}$ in the following result.
\begin{Prop}\label{propeqV}
For any $\delta > 0, \epsilon > 0$, $V_{\delta, \epsilon}$ is a solution of 
\begin{equation}\label{eqV}
\begin{cases}
\begin{aligned}
\partial_t V + &F((Id - \delta V), V)\cdot \nabla_x V =\\
 &\big{(} G( (Id - \delta V), V) - \lambda \big{[} V - T^*\circ V \circ (Id - \delta V)^{-1} \circ T \circ (Id - \delta V) \big{]}  \big{)}(I - \delta \nabla_x V); \end{aligned}\\
V(0,x) = \frac{1}{\delta + \epsilon}(x-x_0).
\end{cases}
\end{equation}
\end{Prop}
\begin{proof}
This claim easily follows from the chain rule together with the fact that $Id + \delta U_{\epsilon} = (Id - \delta V_{\delta, \epsilon})^{-1}$.
\end{proof}
Looking at the the system (\ref{eqV}), it is easy to imagine how we can pass to the limit $\epsilon$ goes to $0$ at the level of the Yosida approximation. We explain this passage to the limit in the next result.
\begin{Prop}\label{convyo}
For any $\delta > 0$, the sequence $(V_{\delta,\epsilon})_{\epsilon > 0}$ converges locally uniformly toward a function $V_{\delta}$, solution of 
\begin{equation}\label{reglim}
\begin{cases}
\begin{aligned}
\partial_t V + &F((Id - \delta V), V)\cdot \nabla_x V =\\
 &\big{(} G( (Id - \delta V), V) - \lambda \big{[} V - T^*\circ V \circ (Id - \delta V)^{-1} \circ T \circ (Id - \delta V) \big{]}  \big{)}(I - \delta \nabla_x V); \end{aligned}\\
V(0,x) = \frac{1}{\delta}(x-x_0).
\end{cases}
\end{equation}
The function $V_{\delta}$ is such that $(Id- \delta V_{\delta}(t))^{-1}$ is well defined for $t> 0$.
\end{Prop}
\begin{proof}
First, let us note that from basic properties of the Yosida approximation, the following holds for any $\delta >0, \epsilon > 0$ and $t> 0$:
\[
\|D_x V_{\delta,\epsilon}\| \leq \frac{1}{\delta}.
\]
In the case $\lambda = 0$, this estimate is enough to gain compactness on the sequence $(V_{\delta,\epsilon})_{\epsilon > 0}$ using equation (\ref{eqV}).

Let us now turn to the case $\lambda > 0$. From theorem \ref{effetreg}, we know that $(Id - \delta V_{\delta,\epsilon}(t))^{-1} = Id + \delta U_{\epsilon}(t)$ is $Ct^{-1}$ lipschitz for some constant $C>0$ and time $t\leq t_f$ where $t_f$ is independent of $\delta$ and $\epsilon$. On the other hand, it is easy to check that $Id - \delta V_{\delta, \epsilon}(t)$ is $C \epsilon (\epsilon + \delta)^{-1} t$ lipschitz for time $t\leq t_f$ where $t_f$ does not depend on $\epsilon$. Furthermore, the right hand side of (\ref{eqV}) stays bounded uniformly near $(t,x) = (0,x_0)$. Thus from the fact that $V_{\delta, \epsilon}$ satisfies (\ref{eqV}), we deduce compactness on the sequence $(V_{\delta,\epsilon})_{\epsilon > 0}$. Extracting a subsequence if necessary, we note $V_{\delta}$ the limit of $(V_{\delta,\epsilon})_{\epsilon > 0}$. The fact that $((Id - \delta V_{\delta, \epsilon}(t))^{-1})_{\epsilon> 0}$ converges toward $(Id - \delta V_{\delta}(t))^{-1}$ for $t>0$ is a simple exercise that we leave to the interested reader.

\end{proof}
\subsection{The limit master equation}
The previous result characterizes the behavior of the Yosida approximation of the solution of the master equation, in particular it gives information on the behavior of the solution near $\{t = 0\}$. This leads us to the following definition, that we comment below.
\begin{Def}\label{defrd}
A solution of the master equation associated to the planning problem characterized by $F,G,\lambda,T,x_0$ is a function $U : (0,T)\times \mathbb{R}^d\to \mathbb{R}^d$ solution of 
\begin{equation}\label{ppme}
\partial_t U + (F(x,U)\cdot \nabla_x)U + \lambda (U - (DT)^*U(Tx)) = G(x, U) \text{ in } (0,\infty)\times\mathbb{R}^d;
\end{equation}
such that $(U(t))_{t \geq 0}$ converges toward $A_{x_0}$ in the sense of graphs as $t$ tends to $0$.
\end{Def}
Let us recall that a sequence of multivalued operator $(A_n)_{n \in \mathbb{N}}$ converges toward a multivalued operator $A$ in the sense of graph if for any sequence $(x_n,y_n)_{n \in \mathbb{N}}$ which converges toward $(x,y) \in H^2$ such that $y_n \in A(x_n)$ for all $n \geq 0$, the property $y \in A(x)$ holds.
We do not comment on the fact that $U$ is a solution of (\ref{ppme}) in $(0,\infty)\times\mathbb{R}^d$. Let us note that the initial condition we impose is rather weak, but, as we shall see in theorem \ref{unique}, it is sufficient to establish uniqueness. In some sense, this condition is strong enough to capture the idea that trajectories of finite cost induced by the solution of the problem necessary start from $x_0$. Moreover let us note that the convergence in the sense of graphs is natural to be expected as the sequence of initial conditions in the penalized problem converges in the sense of graphs toward $A_{x_0}$.

We now present a result of existence and one of uniqueness for such solutions.
\begin{Theorem}\label{exist}
Under the assumptions of theorem \ref{effetreg}, there exists a solution $U$ of the planning problem master equation in the sense of definition \ref{defrd}. 
\end{Theorem}
\begin{proof}
We consider the sequence $(U_{\epsilon})_{\epsilon > 0}$ of solutions of (\ref{penalpp}) and for some $\delta> 0$, $(V_{\delta, \epsilon})_{\epsilon > 0}$ the corresponding Yosida approximations. Thanks to proposition \ref{convyo}, extracting a subsequence if necessary, $(V_{\delta, \epsilon})_{\epsilon > 0}$ converges toward a function $V$ such as in proposition \ref{convyo}. We define the function $U$ by 
\begin{equation}\label{definv}
U(t,x) = V(t,(Id - \delta V(t,\cdot))^{-1}(x)); t> 0; x \in \mathbb{R}^d.
\end{equation}
The fact that $U$ solves the PDE for $t>0$ is a simple consequence of the chain rule. Let us now analyse the behavior of $U$ near $t= 0$. Let us take a real sequence $(t_n)_{n \geq 0}$ converging toward $0$ and two converging sequence of $\mathbb{R}^d$ $(x_n)_{n \geq 0}$ and $(y_n)_{n \geq 0}$ such that for all $n \geq 0$ $y_n = U(t_n, x_n)$. Let us define for all $n \geq 0$, $z_n = (Id - \delta V(t_n, \cdot))^{-1}(x_n)$. By definition of $U$, we deduce that $(V(t_n,z_n))_{n \geq 0} = (y_n)_{n \geq 0}$ and thus that it is a converging sequence. Hence, $(z_n)_{n \geq 0} = (x_n + \delta V(t_n,z_n))_{n \geq 0}$ is also a converging sequence. From the behavior of $V$ near $t = 0$, we deduce finally that $(x_n)_{n \geq 0}$ converges toward $x_0$. Thus $(U(t))_{t \geq 0}$ converges toward $A_{x_0}$ in the sense of graphs as $t$ tends to $0$.
\end{proof}

\begin{Theorem}\label{unique}
Under the assumptions of theorem \ref{effetreg}, there is a unique $U$ solution of the planning problem master equation in the sense of definition \ref{defrd}. Moreover, $U$ is the limit of the sequence $(U_{\epsilon})_{\epsilon > 0}$ of solutions of (\ref{penalpp}).
\end{Theorem}
\begin{proof}
Denote by $U$ and $V$ two solutions of the problem. The main argument of this proof consists in showing that the real function $W$ defined for all $t>0$, $x,y \in \mathbb{R}^d$ by 
\begin{equation}
W(t,x,y) = \langle U(t,x) - V(t,y), x- y\rangle
\end{equation}
is positive on $(0,\infty)\times \mathbb{R}^{2d}$. This function satisfies
\begin{equation}
\partial_t W + F(x,U)\cdot \nabla_x W + F(y,V)\cdot \nabla_y W + \lambda (W- W(t,Tx,T\xi)) \geq 0 \text{ on } (0,\infty)\times \mathbb{R}^{2d}.
\end{equation}
Let us note that from the convergence of $V$ and $U$ toward $A_{x_0}$ as $t$ tends to $0$, we deduce that for all $x,y\in \mathbb{R}^d$, 
\begin{equation}
\liminf_{t \to 0} W(t,x,y) \geq 0.
\end{equation}
Thus we obtain from the maximum principle result (see lemma $3$ in appendix of \citep{bertucci2019some}) that $W \geq 0$ everywhere. This yields that $U = V$. The proof of the previous result guarantees that this solution is the limit of the solutions of the penalized problems.

\end{proof}

\subsection{Links with the induced trajectories}
In this section, we indicate why the trajectories induced by the solution $U$ of the master equation of the planning problem converge toward the constrained point $x_0$. This convergence, explained in the next result, is a consequence of the behavior of $U$ near $t = 0$ and of the monotonicity of $F$ and $G$. We focus on the deterministic case (i.e. $\lambda = 0$) to avoid some technicalities which are due to the particular choice of noise we made, however the same type of approach can be developed in the stochastic case.

\begin{Prop}
Let $U$ be a solution of the problem in the sense of definition \ref{defrd} and assume that $F$ and $G$ are globally lipschitz and that $\lambda = 0$, then for any $x_1 \in \mathbb{R}^d$ and $t_1 > 0$, the trajectory $(x(t))_{0 < t \leq t_1}$ defined (backwardly) by
\begin{equation}\label{traj}
\begin{cases}
\frac{d}{dt}x(t) =  F(x(t),U(t,x(t)));\\ x(t_1) = x_1;
\end{cases}
\end{equation}
is such that $x(t)$ converges toward $x_0$ as $t$ goes to $0$.
\end{Prop}

\begin{proof}
Let us remark that because of the assumption we made, (\ref{traj}) defines a continuous path $(x(t))_{0 < t \leq t_1}$ in $\mathbb{R}^d$. 
Note that the following holds :
\begin{equation}
\begin{cases}
\frac{d}{dt}U(t,x(t)) =  F(x(t),U(t,x(t)));\\ U(t_1,x(t_1)) = U(t_1,x_1).
\end{cases}
\end{equation}
Thus $(x(t),U(t,x(t)))_{0< t\leq t_1}$ is the unique solution of an ordinary differential equation of the form $\dot{y} = f(y)$ for some lipschitz function $f$. Hence, we deduce that $(x(t),U(t,x(t)))$ is uniformly bounded for $0<t\leq t_1$. From the convergence of $U(t)$ toward $A_{x_0}$ as $t$ tends to $0$, in the sense of graphs, we deduce that 
\begin{equation}
x(t) \to_{t \to 0} x_0.
\end{equation}

\end{proof}

\section{A comment on the case of a restricted domain}
In this section, we briefly discuss how the idea developed in the previous section can be extended to situations involving bounded domains. We indicate an example in which the behavior of the solution of the planing problem is clear but we do not present a thorough study of all the different possible structures. Many different behaviors can be expected and more cases shall be treated in a future work.

\subsection{Main differences with the previous case}
The typical form of a master equation in a domain $\Omega \subset \mathbb{R}^d$ is
\begin{equation}\label{mfgdom}
\partial_t U  + (F(x,U)\cdot \nabla_x)U + \lambda (U - (DT)^*U(Tx)) = G(x, U) \text{ in } (0,T)\times{\Omega};
\end{equation}
without boundary conditions on the boundary of the domain $\partial \Omega$. We refer to \citep{bertucci2020} for precise results on such equations, in particular in the MFG setting the fact that this equation behaves nicely is a rather simple extension of similar results in the whole space.

Let us mention that, when $\partial \Omega$ is smooth, a natural (necessary) condition for (\ref{mfgdom}) to be well-posed when equipped with an initial condition is 
\begin{equation}\label{cond}
\langle F(x, p),n(x)\rangle \leq 0 ; x \in \partial \Omega ; p \in \mathbb{R}^d;
\end{equation}
where $n(x)$ is the unit normal vector to $\partial \Omega$ at $x$. The relation (\ref{cond}) is of course an obstacle to some uniform monotonicity assumption on $F$ in its second variable. 
\subsection{The half space case with vanishing conditions}
In this section, we investigate a the case in which : $\Omega = \{x_1 > 0 \}$ and $F_1$, the first component of $F$ vanishes linearly near $0$. More precisely, we assume that for $x_1 \leq 1$, the following holds
\begin{equation}\label{hyp}
F_1(x,p) = x_1\tilde{F_1}(x,p);
\end{equation}
where $\tilde{F_1}$ is such that $\tilde{F} = (\tilde{F_1},F_2,..., F_d)$ satisfies the assumption of the first part.

For the sake of simplicity, we work in the case $T(x) = (x_1, T'(x_2,..,x_d))$ with $T' \in \mathcal{L}(\mathbb{R}^{d-1},\mathbb{R}^{d-1})$.

To sum up, the problem we are interested in is finding a function $U: (0,\infty)\times \Omega \to \mathbb{R}^d$ such that $U$ is a solution of 
\begin{equation}
\partial_t U  + (F(x,U)\cdot \nabla_x)U + \lambda (U - (DT)^*U(Tx)) = G(x, U) \text{ in } (0,T)\times \{x_1 > 0\};
\end{equation}
and $(U(t))_{t > 0}$ converges toward $A_{x_0}|_{\Omega}$ in the sense of graphs, under the assumption (\ref{hyp}).

Now, let us note that the function $V:(0,\infty)\times \mathbb{R}^d \to \mathbb{R}^d$ defined by 
\begin{equation}\label{defV}
V(t,y_1,y_2,...,y_d) = \begin{cases} U(t,e^{y_1-1}, y_2,...,y_d) \text{ if } y_1 < 1;\\ U(t,y) \text{ else;}\end{cases}
\end{equation}
satisfies a master equation of the type presented in section 2. Thus everything we did above applies immediately. Let us mention that in this case, we obtain form (\ref{defV}) and the behavior of $V$ that the solution $U$ of the master equation in the half space is unbounded near $\{x_1= 0\}$, even for $t> 0$. More precisely, it behaves as $ln(x_1)$ near $\{x_1 = 0\}$, which shows that Lipschitz estimates may not be true in the case of a boundary. Indeed they rely mainly on the monotone assumption on $F$.

\section*{Acknowledgments}
The second and third authors have been partially supported by the Chair FDD (Institut Louis Bachelier). The third author has been partially supported by the Air Force Office for Scientific Research grant FA9550-18-1-0494 and the Office for Naval Research grant N000141712095.

\bibliographystyle{plainnat}
\bibliography{bibremarks}

\end{document}